\documentclass[a4paper,10pt]{article}

\usepackage{ae,lmodern}
\usepackage[english,greek,frenchb,french]{babel}
\usepackage[utf8]{inputenc}  
\usepackage[T1]{fontenc}
\usepackage{url,csquotes}
\usepackage{float}
\usepackage{amsfonts,amssymb,enumerate}
\usepackage{amsmath}
\usepackage{authblk}
\usepackage[shortlabels]{enumitem}
\allowdisplaybreaks[1]
\usepackage{amsthm}
\usepackage{graphicx}
\usepackage{bbm}
\usepackage{listings}
\usepackage{titling}
\usepackage{dsfont}
\usepackage{color}
\usepackage{enumitem}
\usepackage{bmpsize}
\usepackage{multibib}
\newcites{perso}{Publications and preprints}
\usepackage[hidelinks,hyperfootnotes=false]{hyperref}
\newcommand{\E}{\mathbb{E}}

    \newcommand{\Prb}{\mathbb{P}}

		\newcommand{\cO}{\mathcal{O}}

		\newcommand{\cT}{\mathcal{T}}

	\newcommand{\sR}{\mathbb{R}}

    \newcommand{\ep}{\varepsilon}

    \newcommand{\ind}{\mathds{1}}

 \theoremstyle{plain}   
 \newtheorem{thm}{Theorem}[section]
\newtheorem{lem}[thm]{Lemma}
\newtheorem{cor}[thm]{Corollary}
\newtheorem{prop}[thm]{Proposition}
\newtheorem{rk}[thm]{Remark}

\numberwithin{equation}{section}

\title{Subcritical sharpness for multiscale Boolean percolation}
\date{}
\author{Barbara Dembin}

 \affil[1]{D-MATH, ETH Z\"urich,
    Switzerland.}
\begin{document}
\selectlanguage{english}
\maketitle
\begin{abstract}
    We consider a multiscale Boolean percolation on $\sR^d$ with radius distribution $\mu$ on $[1,+\infty)$, $d\ge 2$. The model is defined by superposing the original Boolean percolation model with radius distribution $\mu$ with a countable number of scaled independent copies. The $n$-th copy is a Boolean percolation with radius distribution $\mu|_{[1,\kappa]}$ rescaled by $\kappa^{n}$. We prove that under some regularity assumption on $\mu$, the subcritical phase of the multiscale model is sharp for $\kappa $ large enough. Moreover, we prove that the existence of an unbounded connected component depends only on the fractal part (and not of the balls with radius larger than $1$).
\end{abstract}

\section{Introduction}
\paragraph{Overview}
Boolean percolation was introduced by Gilbert in \cite{gilbert} as a continuous version of  Bernoulli percolation, introduced by Broadbent and Hammersley \cite{BroadbentHammersley}.  We consider a Poisson point process of intensity $\lambda>0$ on $\sR^d$ and on each point, we center a ball of potentially random radius. In Boolean percolation we are interested in the connectivity properties of the occupied set: it is defined as the subset of $\mathbb R^d$ consisting of all the points covered by at least one ball. 
This model undergoes a phase transition in $\lambda$ for the existence of an unbounded connected component of balls. For $\lambda<\lambda_c$, all the connected components are bounded, and for $\lambda>\lambda_c$, there exists at least one unbounded connected component.  


\paragraph{Boolean model}
Let $d\geq 2$. Denote by $\|\cdot\|$ the $\ell_2$-norm on $\sR^ d$.  For $r>0$ and $x\in\sR^d$, set \[\mathrm B^x_r:=\{y\in\sR^d:\,\|y-x\|\leq r\}\quad \text{and}\quad \partial  \mathrm B^x_r:=\{y\in\sR^d:\,\|y-x\|= r\}\] for the closed ball of radius $r$ centered at $x$ and its boundary. For short, we will write $\mathrm B_r$ for $\mathrm B_ r^0$.
For a subset $\eta$ of $ \sR^d\times\sR_+$, we define
\[\cO(\eta):=\bigcup_{(z,r)\in\eta}\mathrm B_r^z.\]
Let $\mu$ be a distribution on $\mathbb R_+$ representing the distribution on the radius. Let $\eta$ be a Poisson point process of intensity $\lambda dz\otimes \mu$ where $dz$ is the Lebesgue measure on $\sR^d$.
Write $\Prb_{\lambda,\mu}$ for the law of $\eta$ and $\E_{\lambda,\mu}$ for the expectation under the law $\Prb_{\lambda,\mu}$. 

We say that two points $x$ and $y$ in $\sR^d$ are connected by $\eta$, if there exists a continuous path in $\cO (\eta)$ that joins $x$ to $y$. We say that two sets $A$ and $B$ are connected if there exists $x\in A$ and $y\in B$ such that $a$ and $b$ are connected by $\eta$. We denote by $\{A\longleftrightarrow B\}$ this event.

Define for every $\lambda\ge 0$ and $\mu$, the probability of percolation $$\theta_{\mu} (\lambda):=\lim_{r\rightarrow \infty}\Prb_{\lambda,\mu}\left( 0\longleftrightarrow \partial \mathrm B_r\right).$$
We define the critical parameter associated to the existence of an infinite connected component:
$$\lambda_c(\mu):=\sup\left\{\lambda\geq 0: \theta_{\mu}(\lambda)=0\right\}.$$
 We will work with measures $\mu$ such that
\begin{align}\label{cond:mu}
\int_{0}^\infty t^dd\mu(t)<\infty.
\end{align}
Hall proved in \cite{hall} that this condition is necessary to avoid that all the space is covered.
Under the minimal assumption \eqref{cond:mu}, Gou\'er\'e proved in \cite{gouere} that $0<\lambda_c(\mu)<\infty$.
We also define the following critical parameter:
\[\widehat \lambda_c(\mu):=\inf\left\{\lambda\geq 0:\,\inf_{r>0}\Prb_{\lambda,\mu}(\mathrm B_r\longleftrightarrow\partial \mathrm B_{2r})>0\right\}.\]
Knowing that $\lambda\le\widehat \lambda_c(\mu) $ enables to do renormalization arguments and deduce a lot of properties (see \cite{DCRT,GouereTheret}). Hence, the equality $\widehat \lambda_c(\mu) =\lambda_c(\mu)$ implies that we have a good control on the subcritical regime. If the equality occurs, we say that we have subcritical sharpness. This equality has been proved under moment condition on $\mu$ (see \cite{ATT18,DCRT,ziesche}) and for almost all power-law distributions (see \cite{dembintassion}).

 \paragraph{Multiscale Boolean percolation}
The model of multiscale Boolean percolation consists of an infinite superposition of independent copies of Boolean percolation at different scales.
Let $\mu$ be a finite distribution on $[1,+\infty)$ that satisfies \eqref{cond:mu}.
Let $\kappa>1$. Let $\lambda>0$. For a set $E\subset \sR^d\times\sR_+$, write $E/\kappa$ for the set $\{x/\kappa,x\in E\}$.
We denote by $$\eta_\kappa(\lambda):=\eta ^{(0)}(\lambda)\cup\bigcup_{i=1}^\infty \frac{1}{\kappa^i }(\eta^{(i)}(\lambda)\cap (\sR^d\times [1,\kappa]))$$
where $(\eta^{(i)}(\lambda))_{i\geq 1}$ are i.i.d. Poisson point process of intensity $\lambda\, dz\otimes \mu$. Note that every point in $\cO(\eta_\kappa(\lambda))$ is almost surely covered. Yet, it does not necessarily imply that there exists an unbounded connected component as it does not prevent the existence of a blocking surface of null Lebesgue measure.

We are interested in the percolation properties of $\cO(\eta_\kappa(\lambda))$. Let $\mu_\kappa$ be the distribution such that $\eta_\kappa(\lambda)$ is a Poisson point process of intensity $\lambda dz\otimes \mu_\kappa$. The distribution $\mu_\kappa$ has an infinite mass but is $\sigma$-finite. We will explicit its expression later.

We will here work under the following assumption
\begin{equation}\label{hyp}
    \exists \kappa_0>1\quad \forall \kappa\ge \kappa_0 \qquad \sup_{a\ge \kappa}\sup_{r\ge 1}\frac{a^d\mu([a r,a\kappa])}{\mu([r,\kappa])}\le 1\,
\end{equation}
with the convention $0/0=0$.
This assumption is in particular satisfied for distributions with compact support or distributions of the form $f(r)r^{-(d+1+\delta)}\ind_{r\ge 1}dr$ where $f$ is a non-increasing function such that $0<\inf f <\sup f <\infty$ and $\delta>0$.
The following theorem is the main result of the paper. It states that there is subcritical sharpness for the fractal distribution $\mu_\kappa$ and that the existence of an unbounded connected component does not depend on the large balls.
\begin{thm}\label{thm:main} Let $\mu$ that satisfies assumption \eqref{hyp}. Let $\kappa_0$ be as in \eqref{hyp}. For any $\kappa\ge\kappa_0$, we have \[\lambda_c(\mu_\kappa) =\widehat\lambda_c(\mu_\kappa)=\lambda_c(\mu_\kappa|_{[0,1]}).\]
\end{thm}

 \paragraph{Idea of the proof}
The proof relies on the following key observation. Thanks to condition \eqref{hyp}, for $\kappa\ge \kappa_0$, we can prove that
the Poisson model with intensity $\lambda dz\otimes \mu_\kappa|_{[0,1]}$ stochastically dominates the Poisson model with intensity $\lambda dz\otimes \mu_\kappa|_{[0,\kappa^j]}$ rescaled by $\kappa^j$. Since the support of the distribution $ \mu_\kappa|_{[0,1]}$ is bounded, it is possible to prove subcritical sharpness for this distribution using the standard $\varphi_p(S)$ argument introduced by Duminil-Copin--Tassion in \cite{D-CT} in the context of standard percolation and generalized in the context of Boolean percolation by Ziesche \cite{ziesche}. Using this argument, we can prove that when $\lambda<\lambda_c( \mu_\kappa|_{[0,1]})$, there is exponential decay of the probability of connection. Together with the stochastic domination, we can prove that when $\lambda<\lambda_c( \mu_\kappa|_{[0,1]})$ we have 
\[\inf_{r>0}\Prb_{\lambda,\mu_\kappa}(\mathrm B_r\longleftrightarrow\partial \mathrm B_{2r})=0\]
and $\lambda<\widehat \lambda_c(\mu_\kappa)$. This yields $\lambda_c( \mu_\kappa|_{[0,1]})\le \widehat \lambda_c(\mu_\kappa)\le\lambda_c(\mu_\kappa)$.  The coincidence of these three critical points follows from the previous inequality together with  $\lambda_c( \mu_\kappa|_{[0,1]})\ge\lambda_c(\mu_\kappa)$.

\paragraph{Background}
In previous works on multiscale Boolean percolation, a slightly different definition was used. Define for $\kappa\ge 1$
$$\widetilde \eta _\kappa(\lambda):=\eta ^{(0)}(\lambda)\cup\bigcup_{i=1}^\infty \frac{\eta^{(i)}(\lambda)}{\kappa^i }$$
where $(\eta^{(i)}(\lambda))_{i\geq 1}$ are i.i.d. Poisson point process of intensity $\lambda\, dz\otimes \mu$. Let $\widetilde\mu_\kappa$ be the distribution such that $\widetilde\eta_\kappa(\lambda)$ is a Poisson point process of intensity $\lambda dz\otimes\widetilde \mu_\kappa$. 
With this definition, the range of the radius of the different scaled copies are no longer disjoint, the condition \eqref{cond:mu} is not enough to ensure that the multiscale Boolean model exhibits a non-trivial phase transition. Gouéré proved in \cite{Gouere2009} that $\lambda_c(\widetilde\mu_\kappa)>0$ if and only if \begin{equation}\label{cond:mu2}
    \int_{t\ge 1}t^d\log(t)d\mu(t)<\infty.
\end{equation}
If this condition is not satisfied, the balls with radius greater than $1$ have an infinite mass and $\lambda_c(\widetilde\mu_\kappa)=0$.
\begin{rk}Note in our definition of multiscale percolation, the range of radius among the different scaled copies are disjoint. This enables to remove assumption \eqref{cond:mu2}.
\end{rk}

The Boolean multiscale model was first studied for the distribution $\mu=\delta_1$ by Menshikov--Popov--Vachkovskaia in \cite{Menshikov2001}. They proved that for $\lambda<\lambda_c(\delta_1)$ and $\kappa$ large enough the multiscale model does not percolate.

They later extended in \cite{Menshikov2003} their result to more general distribution $\mu$ that satisfy the following self-similarity condition
\begin{equation*}
    \lim_{a\rightarrow \infty}\sup_{r\ge 1}\frac{a^d\mu([a r,+\infty))}{\mu([r,+\infty))}=0\,
\end{equation*}
and for $\lambda>0$ such that
\begin{equation}\label{cond:lambda}
    \lim_{r\rightarrow\infty }r^d\Prb_{\lambda,\mu}(\mathrm B_r\longleftrightarrow\partial \mathrm B_{2r})=0.
\end{equation}
Note that the condition \eqref{cond:lambda} is quite restrictive since for distributions $\mu$ with an infinite $2d$-moment, there exists no such positive $\lambda$.

The condition \eqref{cond:lambda} was relaxed later by Gouéré in \cite{Gouere2014}, who proved that under the assumption \eqref{cond:mu2}, for $\lambda<\widehat\lambda_c(\mu)$ and $\kappa$ large enough, the multiscale model does not percolate.

\section {Proofs}
\subsection{Proof of Theorem \ref{thm:main}}
In this section, we prove the main theorem. We will need the two following propositions.
This proposition is an adaptation of \cite{ziesche}, the only difference is that the intensity is not finite but locally finite.
\begin{prop}\label{prop:ziesche}Let $\kappa>1$ and $\lambda<\lambda_c(\mu_\kappa|_{[0,1]}) $. There exists $c_\kappa>0$ depending on $\kappa$ and $\lambda$ such that 
\begin{equation}\label{eq:ineqziesche}
 \Prb_{\lambda,\mu_\kappa|_{(0,1]}}(\mathrm B_{1}\longleftrightarrow \partial\mathrm B_{l})\leq \exp(-c_\kappa l)
\end{equation}
\end{prop}
The following proposition is the key observation to prove subcritical sharpness.
\begin{prop}\label{prop}Let $\mu$ that satisfies hypothesis \eqref{hyp}. Let $\kappa\ge \kappa_0$. We have for any $j\geq 1$, $l>1$, $\lambda\geq 0$
\begin{align*}
\Prb_{\lambda,\mu_\kappa|_{[0,\kappa^j]}}(\mathrm B_{\kappa^j}\longleftrightarrow \partial\mathrm B_{l\kappa^j})\leq\Prb_{\lambda,\mu_\kappa|_{(0,1]}}(\mathrm B_1\longleftrightarrow \partial\mathrm B_l).
\end{align*}
\end{prop}
Before proving these two propositions, let us prove the main theorem.
\begin{proof}[Proof of Theorem \ref{thm:main}]
 Let $\lambda<\lambda_c(\mu_\kappa|_{(0,1]})$.   Let $j,l\geq 1$. We have
\begin{align}\label{eq:1}
\Prb_{\lambda,\mu_\kappa}(\mathrm B_{l\kappa ^j}\longleftrightarrow \partial\mathrm B_{2l\kappa ^j})&\leq \Prb_{\lambda,\mu_\kappa|_{(0,\kappa^j ]}}(\mathrm B_{l\kappa ^j}\longleftrightarrow \partial\mathrm B_{2l\kappa ^j})\nonumber\\&\quad+\Prb_{\lambda,\mu_\kappa}\left(\exists (x,r)\in\eta_\kappa (\lambda):\,r\geq \kappa^j, \,\mathrm B_r^x\cap \mathrm B_{2l\kappa ^j}\neq\emptyset\right).
\end{align}
Let us start by estimating the second term in the inequality:
\begin{align*}
\Prb_{\lambda,\mu_\kappa}\left(\exists (x,r)\in\eta_\kappa (\lambda):\,r\geq \kappa^j, \,\mathrm B_r^x\cap \mathrm B_{2l\kappa ^j}\neq\emptyset\right)=1-\exp(-\lambda dz\otimes\mu(E))
\end{align*}
where $E:=\{(x,r): \|x\|_2\leq 2l\kappa ^j+r,\,r\geq \kappa ^j\}$.
We have
\begin{align*}
dz\otimes\mu(E)&=\int_{r\geq \kappa ^ j}\alpha_d (2l\kappa^j +r)^dd\mu(r)\le \alpha_d (4l)^d\int_{r\ge \kappa^j}r^dd\mu(r)
\end{align*}
where $\alpha_d$ is the volume of the unit ball in $\sR^d$.
It yields that
\begin{equation}\label{eq:2}
\Prb_{\lambda,\mu_\kappa}\left(\exists (x,r)\in\eta_\kappa (\lambda):\,r\geq \kappa^j, \,\mathrm B_r^x\cap \mathrm B_{2l\kappa ^j}\neq\emptyset\right)\leq \lambda \alpha_d (4l)^d\int_{r\ge \kappa^j}r^dd\mu(r).
\end{equation}
Let us now control the first term.
There exists a constant $c_d$ depending only on $d$ such that we can cover $\partial \mathrm B_{l\kappa^j}$ by at most $c_d l^{d-1}$ balls of radius $\kappa^j$ centered at $\partial \mathrm B_{l\kappa^j}$. By union bound, we get
\begin{equation}\label{eq:3}
\begin{split}
\Prb_{\lambda,\mu_\kappa|_{(0,\kappa^j ]}}(\mathrm B_{l\kappa ^j}\longleftrightarrow \partial\mathrm B_{2l\kappa ^j})&\leq 
 c_dl^{d-1}\Prb_{\lambda,\mu_\kappa|_{(0,\kappa^j]}}(\mathrm B_{\kappa^j}\longleftrightarrow \partial\mathrm B_{l\kappa^j})\\
&\leq c_d l^{d-1}\exp(-c_\kappa l)
\end{split}
\end{equation}
where we use in the last inequality Propositions \ref{prop} and \ref{prop:ziesche}.
Combining inequalities \eqref{eq:1}, \eqref{eq:2} and \eqref{eq:3}, we obtain 
\begin{align*}
\Prb_{\lambda,\mu_\kappa}(\mathrm B_{l\kappa ^j}\longleftrightarrow \partial\mathrm B_{2l\kappa ^j})\leq c_d l^{d-1}\exp(-c_\kappa l)+\lambda \alpha_d (4l)^d\int_{r\ge \kappa^j}r^dd\mu(r).
\end{align*}
Let $\ep>0$. We first choose $l$ large enough depending on $c_\kappa$ and $\ep$ and then $j$ large enough depending on $\kappa$, $\ep$ and $l$ so that
\begin{align*}
\Prb_{\lambda,\mu_\kappa}(\mathrm B_{l\kappa ^j}\longleftrightarrow \partial\mathrm B_{2l\kappa ^j})\leq \ep
\end{align*}
where we recall that since $\mu$ has a finite $d$-moment 
\[\lim_{j\rightarrow\infty}\int_{r\ge \kappa^j}r^dd\mu(r)=0.\]
It follows that
\begin{align*}
\inf_{r>0}\Prb_{\lambda,\mu_\kappa}(\mathrm B_r\longleftrightarrow\partial \mathrm B_{2r})=0
\end{align*}
and $\lambda\leq \widehat \lambda_c(\mu_\kappa)$.
Hence, \[ \widehat \lambda_c(\mu_\kappa)\geq \lambda_c(\mu_\kappa|_{(0,1]})\geq \lambda_c(\mu_\kappa).\]
The result follows from the fact that $\widehat \lambda_c(\mu_\kappa)\le \lambda_c(\mu_\kappa)$.
\end{proof} 
\subsection{Proof of Propositions \ref{prop:ziesche} and \ref{prop}}
Let $m>0$. Set $h_m$ be the contraction by $m$ that is $h_m(x):=x/m$ for $x\in\sR$.
Set $$\cT_m\mu:=m^dh_m*\mu$$ where $h_m*\mu$ is the pushforward of $\mu$ by $h_m$. We will need the following Lemma that characterized the distribution of a contracted in space Poisson point process.
\begin{lem}\label{lem:rescale}Let $m>0$ and $\lambda>0$. Let $\nu$ be a distribution on $\sR_+$. Let $\eta$ be a Poisson point process of intensity $\lambda dz\otimes \nu$.
Then $\eta/m$ is a Poisson point process of intensity $\lambda dz\otimes \cT_m\nu$.
\end{lem}
From this lemma, we can deduce the following straightforward corollary.
\begin{cor}\label{cor}Let $\kappa\ge 1$. We have $$\mu_\kappa=\mu+\sum_{j=1}^\infty \cT_{\kappa^j}\mu|_{[1,\kappa]}.$$

\end{cor}
\begin{proof}[Proof of Lemma \ref{lem:rescale}]It is clear that $\eta/m$ is still a Poisson point process, we only need to prove that its intensity is $\lambda dz\otimes \cT_m\nu$. Let $E\subset \sR^d\times \sR_+$. We claim that 
\begin{equation}
(dz\otimes \nu)(mE)=(dz\otimes \cT_m\nu)(E).
\end{equation}
Indeed, we have
\begin{align*}
(dz\otimes \nu)(mE)=\int_{(z,r)\in mE}dz d\nu(r)&=\int_{(mz,mr)\in mE} m^{d}dzd\nu(r/m)\\&=\int_{(z,r)\in E}dzd\cT_m\nu(r)=(dz\otimes \cT_m\nu)(E).
\end{align*}
\end{proof}

Thanks to Corollary \ref{cor}, we can now prove Proposition \ref{prop}.

\begin{proof}[Proof of Proposition \ref{prop}]Thanks to Lemma \ref{lem:rescale}, we have for $l>1$ and $j\ge 0$
\begin{equation*}
\Prb_{\lambda,\mu_\kappa|_{(0,\kappa^j]}}(\mathrm B_{\kappa^j}\longleftrightarrow \partial\mathrm B_{l\kappa^j})=\Prb_{\lambda,\cT_{\kappa^j}\mu_\kappa|_{(0,\kappa^j]}}(\mathrm B_{1}\longleftrightarrow \partial\mathrm B_{l}).
\end{equation*}
To complete the proof, let us prove the following inequality
\begin{align*}
\Prb_{\lambda,\cT_{\kappa^j}\mu_\kappa|_{(0,\kappa^j]}}(\mathrm B_{1}\longleftrightarrow \partial\mathrm B_{l})\leq \Prb_{\lambda,\mu_\kappa|_{(0,1]}}(\mathrm B_{1}\longleftrightarrow \partial\mathrm B_{l}) .
\end{align*}
Using Corollary \ref{cor}, we have
\begin{align*}
\cT_{\kappa^j}\mu_\kappa|_{(0,\kappa^j]}&=\cT_{\kappa^j}\mu|_{[1,\kappa^j]}+\sum_{k=1}^\infty\cT_{\kappa^j}\cT_{\kappa^k}\mu|_{[1,\kappa]}\\
&= \sum_{k=1}^{j}\cT_{\kappa^{k}}\cT_{\kappa^{j-k}}\mu|_{[\kappa^{j-k},\kappa ^{j-k+1}]}+\sum_{k=j+1}^\infty\cT_{\kappa^k}\mu|_{[1,\kappa]}
\end{align*}
Let us prove that for any $k\ge 1$ $\cT_{\kappa^{k}}\mu|_{[\kappa^{k},\kappa ^{kk+1}]}\preceq \mu|_{[1,\kappa]}$ where we write $\mu\succeq \nu$ when $\mu$ stochastically dominates $\nu$ (for every $r>0$, we have
$\mu([r,+\infty))\ge \nu([r,+\infty))$).
Let $\kappa_0$ be as in hypothesis \ref{hyp}. Let $\kappa\ge \kappa_0$.
By hypothesis \eqref{hyp}, we have for $r\in[1,\kappa]$
\[\cT_{\kappa ^{k}}\mu|_{[\kappa^{k},\kappa^{k+1}]}([r,\kappa])=\kappa ^{dk}\mu([\kappa^{k}r,\kappa ^{k+1}]\le \mu([r,\kappa]).\]
It yields that 
\begin{equation*}
\cT_{\kappa^j}\mu_\kappa|_{(0,\kappa^j]}
\preceq  \sum_{k=1}^{j}\cT_{\kappa^{k}}\mu|_{[1,\kappa]}+\sum_{k=j+1}^\infty\cT_{\kappa^k}\mu|_{[1,\kappa]}=\mu_\kappa|_{(0,1]}.
\end{equation*}
Hence, we have
\begin{align*}
\Prb_{\lambda,\cT_{\kappa^j}\mu_\kappa|_{(0,\kappa^j]}}(\mathrm B_{1}\longleftrightarrow \partial\mathrm B_{l})\leq \Prb_{\lambda,\mu_\kappa|_{(0,1]}}(\mathrm B_{1}\longleftrightarrow \partial\mathrm B_{l}) .
\end{align*}
This yields the proof.
\end{proof}

Finally, let us explain how the proof of Ziesche \cite{ziesche} can be extended in the general case of $\sigma$-finite measure (Proposition \ref{prop:ziesche}).
\begin{proof}[Sketch of the proof of Proposition \ref{prop:ziesche}]
First note that $\lambda ds\otimes \mu_\kappa$ is a $s$-finite measure on $\sR^ d\times \sR_+\setminus\{0\}$ (hence $\sigma$- finite), that is, it can be written as a countable sum of finite measures. The Mecke equation (see Theorem 4.1 in \cite{last_penrose_2017}) and the Margulis-Russo formula (see \cite{Last14}) both hold for intensity measures that are $s$-finite. Denote by $\mathcal B(\sR^ d)$ the Borelian subsets of $\sR^ d$. For each $S\in\mathcal B(\sR^ d)$ such that $\mathrm B_1\subset S$, we define
\begin{equation}
\varphi_{\lambda}(S):=\lambda\int_{r\in(0,1]}\int_{z\in\sR^d}\mathbf{1}_{\mathrm B_r^z\cap\partial S\neq\emptyset}\,\Prb_{\lambda,\mu|_{(0,1]}}\left(\mathrm B_1\stackrel{\cO(\{(w,s)\in\eta: \mathrm B_s^w\subset S\})}{\longleftrightarrow }\mathrm B_r^z\right) dz\,d\mu_\kappa(r).
\end{equation}
This corresponds to the expected number of open balls intersecting the boundary of $S$ that are connected to $\mathrm B_1$ inside $S$.
The arguments of Ziesche hold in that context, in particular, when $\lambda<\lambda_c(\mu_\kappa|_{[0,1]})$, there exists $S\in\mathcal B(\sR^ d)$ such that $\mathrm B_1\subset S$ and $\varphi_{\lambda}(S)<1$. We conclude the existence of $c_\kappa>0$ depending on $\kappa $ and $\lambda$ such that inequality \eqref{eq:ineqziesche} holds.

\end{proof}

\paragraph{Acknowledgements}
The author would like to thank Vincent Tassion for fruitful discussions that initiated this project.
This project has received funding from the European Research Council (ERC) under the European Union’s Horizon 2020 research and innovation program (grant agreement No
851565). 
\bibliographystyle{plain}
\bibliography{biblio}
\end{document}